\documentclass[11pt]{amsart}
\usepackage{amsmath,amssymb,amscd,mathrsfs,amscd}
\usepackage[curve,matrix,arrow]{xy}

\usepackage{hyperref}
\hypersetup{
  colorlinks=true, 
   linkcolor=red,  
}

\renewcommand{\theenumi}{\alph{enumi}}

\DeclareMathOperator{\Ext}{Ext}
 \DeclareMathOperator{\soc}{soc}
 \DeclareMathOperator{\head}{head}
 \newcommand{\socle}{\soc}
 
 \DeclareMathOperator{\car}{char}

\newcommand{\alt}{\alpha}

\newcommand{\Nr}[2]{\dim (\Bil{#2})^{#1}}

  \DeclareMathOperator{\Lie}{Lie}
  \DeclareMathOperator{\Id}{Id}
  
  \renewcommand{\H}{\mathrm{H}}

\DeclareMathOperator{\sign}{sign}
 
 \newcommand{\kx}{k^\times}
 
 \newcommand{\at}{\widetilde{\alpha}}   
 \newcommand{\hsr}{\alpha_{0}}         
 
 \newcommand{\qb}{\bar{q}}
  
 \newcommand{\br}{b_{(r)}}
 \newcommand{\bz}{b_{\Z}}
 \newcommand{\bzr}{(\bz)_{(r)}}
 \newcommand{\Vz}{V_\Z}

  \newcommand{\qr}{q_{(r)}}

 \newcommand{\onto}{\twoheadrightarrow}

 \newcommand{\Gm}{\mathbb{G}_m}
 
 \newcommand{\B}{\mathscr{B}}

 \newcommand{\g}{\mathfrak{g}}
 \newcommand{\gt}{\tilde{\g}}
 
 \newcommand{\la}{\lambda}
 \newcommand{\s}{\sigma}
 
 \newcommand{\iso}{\xrightarrow{\sim}}
 
 \newcommand{\hch}{h^{\vee}}
 
 \newcommand{\roots}{\Phi}
 
 \newcommand{\Gt}{\tilde{G}}
 
 \newcommand{\Vb}{\bar{V}}
  \newcommand{\fb}{\bar{f}}
  \renewcommand{\qb}{\bar{q}}

 \DeclareMathOperator{\Bil}{Bil}
  \DeclareMathOperator{\Quad}{Quad}
 
 \newcommand{\I}{\mathscr{I}}

 \newcommand{\qform}[1]{{\left\langle{#1}\right\rangle}}                   
\newcommand{\darkrad}{0.115}

 \DeclareMathOperator{\Span}{span}
  \DeclareMathOperator{\Hom}{Hom}

\DeclareMathOperator{\rrad}{rad}
\DeclareMathOperator{\im}{im}
 \renewcommand{\O}{\mathrm{O}}

\newcommand{\cprod}{{\mathinner{\mkern2mu\raise2.5pt\hbox{.}}}}

\newcommand{\N}{\mathbb{N}}

\newcommand{\bhat}{\hat{b}}

\renewcommand{\g}{\mathfrak{g}}

\newcommand{\FF}{\mathbb{F}}
\newcommand{\Z}{\mathbb{Z}}
\newcommand{\Q}{\mathbb{Q}}
\newcommand{\C}{\mathbb{C}}

\newcommand{\ot}{\otimes}

\DeclareMathOperator{\SL}{SL}
\DeclareMathOperator{\GL}{GL}
\DeclareMathOperator{\SO}{SO}
\DeclareMathOperator{\Sp}{Sp}
\DeclareMathOperator{\Spin}{Spin}

\newtheorem{theorem}[equation]{Theorem}

\newtheorem{lemma}[equation]{Lemma}
\newtheorem{lem}[equation]{Lemma}

\newtheorem{prop}[equation]{Proposition}

\newtheorem{cor}[equation]{Corollary}

\theoremstyle{definition}
\newtheorem{defn}[equation]{Definition}
\newtheorem{example}[equation]{Example}
\newtheorem{eg}[equation]{Example}
\newtheorem{remark}[equation]{Remark}

\numberwithin{equation}{subsection}

\begin{document}
 \title[Bilinear and quadratic forms on modules of reductive groups]{Bilinear and quadratic forms on rational modules of split reductive groups}
 
 \begin{abstract}
 The representation theory of semisimple algebraic groups over the complex numbers (equivalently, semisimple complex Lie algebras or Lie groups, or real compact Lie groups) and the question of whether a 
 given representation is symplectic or orthogonal has been solved over the complex numbers since at least the 1950s. Similar results for Weyl modules of split reductive groups over fields of characteristic different from 2 hold by 
 using similar proofs.  This paper considers analogues of these results for simple, induced and tilting modules of split reductive groups over fields of prime characteristic as well as a complete answer for Weyl modules over fields of characteristic~2. 
 \end{abstract}
 
 \author{Skip Garibaldi}
 \thanks{The first author was partially supported by NSF grant DMS-1201542 and the Charles T.~Winship Fund at Emory University.}
\address{Garibaldi: Institute for Pure and Applied Mathematics, UCLA, 460 Portola Plaza, Box 957121, Los Angeles, California 90095-7121, USA}
\email{skip@member.ams.org}
\urladdr{http://www.garibaldibros.com/}

\author{Daniel K. Nakano}
\thanks{The second author was partially supported by NSF grant DMS-1402271}
\address{Nakano: Department of Mathematics, University of Georgia,
Athens, Georgia~30602, USA}
\email{nakano@math.uga.edu}

\thanks{{\tt Version of \today.}}

\maketitle

\setcounter{tocdepth}{1}
\tableofcontents

\setlength{\unitlength}{.75cm}

\section{Introduction}

\subsection{} The representation theory of semisimple algebraic groups over the complex numbers --- equivalently, semisimple complex Lie algebras or Lie groups --- is well known: the set of isomorphism classes of irreducible representations of the simply connected cover of a group $G$ is in bijection with the cone of dominant weights of the root system of $G$.  

Classifying representations of $G$ over $\C$ is the same as classifying homomorphisms from $G$ into a group of type $A$.  One can equally well ask about homomorphisms into other groups, for which homomorphisms into groups of type $B$ or $D$ (orthogonal representations) and groups of type $C$ (symplectic representations) play a distinguished role.  The basics of this theory were laid out in \cite{Malcev}, were known to Dynkin \cite[as summarized in pp.~254, 255, item C]{Dynk:max}, and a complete solution is clearly described in \cite[\S3.2.4]{GW2} or \cite[\S{VIII.7.5}]{Bou:g7}.  One reduces the problem to studying irreducible representations, then giving an algorithm in terms of the dominant weight $\la$ for determining whether the irreducible representation $L(\la)$ has a nonzero $G$-invariant symmetric or skew-symmetric bilinear form.  The algorithm is proved via restricting the representation to a principal $A_1$ subgroup of $G$.  This material is now a plank in the foundations of representation theory, where it has applications to determining the subalgebras of semisimple complex Lie algebras \cite{Dynk:max} and distinguishing real and quaternionic representations of compact real Lie groups \cite[Ch.~IX, App.~II.2, Prop.~3]{Bou:g7}.  For general $k$, invariant bilinear forms are used to provide information about the groups, for example throughout the book \cite{KMRT}, or for bounding the essential dimension of $G$ as in \cite{ChSe} or \cite{BabicCh}, or for controlling Lie subalgebras as in \cite[\S5.8, Exercise 1]{StradeF}.

\subsection{} This paper concerns the generalization of the above theory of symplectic and orthogonal representations to the case of a field $k$ of prime characteristic.  The situation is more complicated. 
There are four classes of representations (each of which is in bijection with the cone of dominant weights) that have a claim to being the natural generalization of the irreducible representation over $\C$ with highest weight $\la$: the irreducible representation $L(\la)$, the induced module $H^0(\la)$, the Weyl module $V(\la)$, and the tilting module $T(\la)$.  The four representations are related by nonzero maps
\begin{equation} \label{triple}
\begin{CD}
V(\la) @>\text{surjective}>> L(\la) \\
@V{\text{injective}}VV @VV{\text{injective}}V \\
T(\la) @>\text{surjective}>> H^0(\la).
\end{CD}
\end{equation}
where all four maps 
are unique up to multiplication by an element of $k^\times$ (and the diagram commutes up to multiplication by a scalar).  The definitions make sense for every field $k$, and when $\car k = 0$ all the maps in \eqref{triple} are isomorphisms.  We refer to \cite{Jantzen} for information on this subject.  In extending the theory of orthogonal and symplectic representations to include the case where $\car k$ is prime, two complications arise.

First, although the theory of bilinear forms on irreducible representations over $\C$ translates easily to irreducible representations and Weyl modules over any $k$ (see Lemma \ref{bil}), it does not directly translate to the representations $H^0(\la)$ and $T(\la)$.  For example, when $V$ is irreducible or Weyl, the space of $G$-invariant bilinear forms is at most 1-dimensional, but for $V = H^0(\la)$ or $T(\la)$ it can be larger.  We give a formula for this dimension in the case of $T(\la)$ in Theorem \ref{tiltingforms} and relate it to $B$-cohomology in the case of $H^0(\la)$ in Theorem \ref{H0.bil}.

Second, when $\car k = 2$, the theory of $G$-invariant bilinear forms is insufficient to treat homomorphisms of $G$ into groups of type $B$ and $D$, because such groups are related to the existence of $G$-invariant \emph{quadratic} forms.  The question of the existence of $G$-invariant quadratic forms on irreducible representations has previously been studied (cf. \cite{Willems77}, \cite{GowWillems:methods}, \cite{SinWillems}), and there is no known, straightforward necessary and sufficient condition, cf.~Section \ref{irred.sec}.  In contrast to this, we completely solve the question of which Weyl modules have a $G$-invariant quadratic form, see Theorem \ref{orth}.

One major feature throughout our paper is the interplay with the admissibility of various $G$-invariant bilinear forms and the rational cohomology of $G$. In the final section (Section \ref{wedge2.coh}), as a byproduct of our results, we are able to calculate $\H^1(G, \Lambda^2(V))$ for many cases where $V$ is $L(\la)$, $T(\la)$, or $H^0(\la)$.

\subsection{} 
One motivation for the present paper stems from \emph{The Book of Involutions}, \cite{KMRT}, which includes ad hoc constructions of $G$-invariant quadratic forms on various irreducible representations for various semisimple groups $G$.  These irreducible representations happen to also be Weyl modules, so our Theorem \ref{orth} provides the theoretical context for the ad hoc constructions by giving a necessary and sufficient condition for the existence of such a form.  Proposition \ref{Sp.odd} gives a new, explicit example of a $G$-invariant quadratic form that might have been constructed in \S10.D of \cite{KMRT} but was not.

This paper treats the case where $G$ is split reductive.  A sequel work will extend these results to the case where $G$ need not be split.

\subsection{Acknowledgements} We thank George McNinch, Anne Qu\'eguiner-Mathieu, and Jean-Pierre Tignol for helpful discussions. 

\smallskip

\section{Background: Representations of split reductive groups} \label{back.sec}

\subsection{Notation} We will follow the notation from \cite{Jantzen} and \cite{Bou:g4}. Throughout we consider a field $k$ of characteristic $p \ge 0$ and an algebraic group $G$ over $k$ (i.e., a smooth affine group scheme of finite type over $k$) that is reductive and split.  If $k$ is separably closed, then every reductive algebraic $k$-group is split.  We fix in $G$ a pinning, which includes the following data:
\begin{itemize}
\item $T$: a $k$-split maximal torus in $G$. 
\item $\Phi$: the  root system of $G$ with respect to $T$. When referring to short and long roots, when a root system 
has roots of only one length, all roots shall be considered as both short and long.
\item $\Pi=\{\alpha_1,\dots,\alpha_n\}$: the set of simple roots. We adhere to the ordering of the simple roots as given in
\cite{Jantzen} (following Bourbaki); see \eqref{C.diag} for the numbering for type $C$. 
\item $\at$: the maximal root. 

\item  $B$: a Borel subgroup containing $T$ corresponding to the negative roots. 
\item  $W = N_G(T)/T$: the Weyl group.
\item  $w_{0}$: longest element in $W$, relative to the choice of simple roots $\Pi$.
\item $P:=\mathbb Z \omega_1\oplus\cdots\oplus{\mathbb Z}\omega_n$: the weight lattice, where the fundamental dominant weights $\omega_i$ are defined by $\langle\omega_i,\alpha_j \rangle=\delta_{ij}$, $1\leq i,j\leq n$.
\item $X(T) = \Hom(T, \Gm) \subseteq P$.
\item $\leq$ on $P$: a partial ordering of weights, for $\lambda, \mu \in P$, $\mu\leq \lambda$ if and only if $\lambda-\mu$ is in $\N \alpha_1 + \cdots + \N \alpha_n$.
\item $P_+ :={\mathbb N}\omega_1+\cdots+{\mathbb N}\omega_n$: the dominant weights.
\item $X(T)_+ := X(T) \cap P_+$.
\item $Q:={\mathbb Z}\Phi$: the root lattice. 

\item $F:G\rightarrow G$: the Frobenius morphism. 
\item $G_r=\text{ker }F^{r}$: the $r$th Frobenius kernel of $G$. 
\item $M^{[r]}$:  the module obtained by composing the underlying representation for 
a rational $G$-module $M$ with $F^{r}$.

\item $H^0(\lambda) := \operatorname{ind}_B^G\lambda$, $\lambda\in X(T)_{+}$: the induced module whose character is provided by Weyl's character formula.  
\item $V(\lambda):=H^{0}(-w_{0}\lambda)^{*}$: the Weyl module. 
\item $L(\lambda)$: the irreducible finite dimensional $G$-module with highest weight $\lambda\in X(T)_{+}$. 
\item $T(\lambda)$: Ringel's indecomposable tilting module corresponding to $\lambda\in X(T)_{+}$, as defined in \cite[Prop.~E.6]{Jantzen}.
\end{itemize} 

Note that $V(\la)^* = H^0(-w_0 \la)$, whereas $L(\la)^* = L(-w_0 \la)$ and $T(\la)^* = T(-w_0 \la)$.

\subsection{Tori} In the case when $G$ is a torus, $\roots = \emptyset$, $Q = 0$, and $P = X(T) = X(T)_+$.  For each $\la \in X(T)_+$, $L(\la) = H^0(\la) = V(\la) = T(\la)$ is one-dimensional.

\subsection{} \label{red.sc} Suppose $G = T_0 \times \prod_{i=1}^n G_i$ for $T_0$ a split torus and $G_i$ simple and simply connected.  Let $T_i$ be a maximal split torus in $G_i$ and take $T = \prod_{i=0}^n T_i$; a pinning of $G$ relative to $T$ is equivalent to fixing a pinning of each $G_i$ relative to $T_i$.  Then $\roots$ is the union of the root systems of the $G_i$.  Setting $P_i$ to be the weight lattice of $G_i$, i.e., $X(T_i)$, we find that $X(T)_+ = P_0 \oplus \bigoplus_i P_i = P_+$. 

\begin{lem} \label{red.tensor}
With the notation of the previous paragraph, for $\la_i \in X(T_i)_+$, we have:
\[
L(\sum \la_i) = \ot L(\la_i), \quad V(\sum \la_i) = \ot V(\la_i), \quad \text{and} \quad H^0(\sum \la_i) = \ot H^0(\la_i).
\]
\end{lem}

\begin{proof}
The claim for $H^0$ is \cite[Lemma I.3.8]{Jantzen}.  Dualizing gives the claim for $V$.  For $L$, an irreducible representation of $G$ is a tensor product of irreducible representations of $T_0$ and of the $G_i$, and inspecting highest weights yields the claim.
\end{proof}

\subsection{} \label{red.gen} For an arbitrary split reductive group $G$, there exists a split torus $T_0$ and split, simple, simply connected groups $G_1, \ldots, G_n$ as above and a central isogeny $\pi \!: T_0 \times \prod G_i \to G$, all of which are in some sense unique.  The quotient $\pi$ relates the chosen pinning of $G$ relative to $T$ to a pinning of $T_0 \times \prod G_i$ relative to the split maximal torus $\pi^{-1}(T)^\circ$ such that $\pi^* X(T)_+ = \prod X(T_i)_+$, and representations of $G$ induce representations of $T_0 \times \prod G_i$.  In this way, when proving the results in this paper, it is harmless to assume that $G$ is as in Section \ref{red.sc}.

In later parts of the paper, $G_1$ will be used to denote the first Frobenius kernel of $G$; the difference will be clear from context.

\section{Symmetric Tensors and Symmetric Powers} \label{symvsym.sec}

\subsection{ } \label{symsubsec} Let $V$ be a $k$-vector space.  

\begin{defn} \label{symp.def}
The symmetric group $\Sigma_n$ on $n$ letters acts on $\ot^n V$ by permuting the indices of a tensor $v_1 \ot \cdots \ot v_n$.  Define $S'_n(V) \subseteq \ot^n V$ via
\[
S'_n(V) = \{ x \in \ot^n V \mid \text{$\s x = x$ for all $\s \in \Sigma_n$} \};
\]
it is the space of \emph{symmetric tensors}.
The symmetrization map $s \!: \ot^n V \to \ot^n V$ is defined by $s(x) = \sum_{\s \in \Sigma_n} \s x$ and
\[
S''_n(V) := \im s \subseteq S'_n(V);
\]
elements of $S''_n(V)$ are \emph{symmetrized tensors}.
Evidently, if $n!$ is not zero in $k$, then $S''_n(V) = S'_n(V)$.
\end{defn}

\begin{example}[Bilinear forms] \label{bil.eg}
The space 
\[
\Bil V := V^* \ot V^*
\]
is the vector space of bilinear forms on $V$, and $S'_2(V^*)$ is the subspace of symmetric bilinear forms.  If $\car k = 2$, then $S''_2(V^*)$ is the span of elements $x \ot y - y \ot x$ for $x, y \in V^*$, i.e., $S''_2(V^*)$ is the space of alternating bilinear forms.

Regardless of the characteristic of $k$, we identify $\ot^2 V^* \iso \Hom(V, V^*)$ denoted $b \mapsto \bhat$ where $\bhat$ is defined by $\bhat_v := b(v, -)$.  Note that, as our vector spaces are finite-dimensional, the canonical inclusion $\ot^n (V^*) \subseteq (\ot^n V)^*$ is an equality by dimension count, so we omit the unnecessary parentheses from this expression.  The (left) radical of $b$ is 
$\rrad b := \ker \bhat$.
\end{example} 

\subsection{Symmetric powers} \label{sympowers}
The \emph{$n$-th symmetric power} $S^n(V)$ of $V$ is the image of $\ot^n V$ in the quotient of $\oplus_{i \ge 0} \ot^i V$ by the 2-sided ideal generated by all elements of the form $v \ot v' - v' \ot v$ for $v, v' \in V$; we write $\rho \!: \ot^n V \onto S^n(V)$ for the quotient map.  One typically omits the symbol $\ot$ when writing elements of $S^n(V)$.  If we fix a basis $v_1, \ldots, v_d$ of $V$, then $S^n(V)$ has basis $\{ v_1^{e_1} v_2^{e_2} \cdots v_d^{e_d} \mid \sum e_i = n \}$.

We obtain a commutative diagram (cf.~\cite[\S{IV.5.8}]{Bou:alg2}) of linear maps
\[
\xymatrix{
\ot^n V \ar[d]^\rho \ar[dr]^s & && & \ot^n V\ar[d]^\rho \\
S^n(V) \ar[r]^\varphi &S''_n(V)  &  \subseteq &S'_n(V) \ar[ur] \ar[r]^\psi & S^n(V)
}
\]
The map $\varphi$ is the multilinearization map $\varphi(v_1 \cdots v_n) = \sum_{\s \in \Sigma_n} v_{\s(1)} \ot \cdots \ot v_{\s(n)}$.  The compositions $\varphi \psi \!: S''_n(V) \to S''_n(V)$ and $\psi \varphi \!: S^n(V) \to S^n(V)$ are multiplication by $n!$.

\begin{example}[Quadratic forms]
The elements of the vector space 
\[
\Quad V  := S^2(V^*)
\]
are the  quadratic forms on $V$.  The multilinearization map $\varphi$ sends a quadratic form $q$ to its \emph{polar bilinear form} $b_q := \phi(q)$ given by
\[
b_q(v,v') = q(v + v') - q(v) - q(v').
\]
If $b_q = 0$, then $q$ is said to be \emph{totally singular}.  The \emph{radical} of $q$ is 
\[
\rrad q := \{ v \in V \mid \text{$v \in \rrad b_q$ and $q(v) = 0$} \}.
\]
Evidently $\rrad q$ is a vector space, and for each extension $K$ of $k$, $\rrad (q \ot K) \supseteq (\rrad q) \ot K$.  If $\car k \ne 2$, then $\rrad q = \rrad b_q$ and the only totally singular quadratic form is the zero form.  

We refer to \cite{EKM} for background concerning quadratic forms over a field, including the case $\car k = 2$.
\end{example}

\subsection{Alternating tensors} \label{alt}

\begin{defn} \label{altp.def} We define subspaces $A''_n(V) \subseteq A'_n(V) \subseteq \ot^n V$ via setting
\[
A'_n(V) := \{ x \in \ot^n V \mid \text{$\s x = (\sign \s) x$ for all $\s \in \Sigma_n$} \}
\]
and $A''_n(V)$ to be the image of the \emph{skew-symmetrization map} $\alt \!: \ot^n V \to A'_n(V)$ defined via $\alt(x) = \sum_{\s \in \Sigma_n} (\sign \s) \cdot \s x$.
These subspaces are analogous to the symmetric tensors in Definition \ref{symp.def}.

Evidently, the restriction of $\alt$ to $A'_n(V)$ acts as multiplication by $n!$.   We may identify $\Lambda^2(V)$ with $A''_2(V)$ and $\Lambda^2(V^*)$ with the space of alternating bilinear forms on $V$ (as $\Lambda^2 (V^*) = \Lambda^2(V)^*$), i.e., the forms $b \in V^* \ot V^*$ such that $b(v,v) = 0$ for all $v \in V$.  If $\car k = 2$, $A'_2(V^*) = S'_2(V^*)$ is the space of symmetric bilinear forms on $V$ and $A''_2(V^*)$ is a proper subspace of $A'_2(V^*)$ for nonzero $V$.
\end{defn}

The sequences
\begin{equation} \label{bil.seq}
\begin{CD}
0 @>>> S'_2(V) @>>> V \ot V @>\alt>> \Lambda^2(V) @>>> 0
\end{CD}
\end{equation}
and
\begin{equation} \label{wedgeS.seq}
\begin{CD}
0 @>>> \Lambda^2(V) @>>> V \ot V @>\rho>> S^2(V) @>>> 0
\end{CD}
\end{equation} 
are exact.
If $\car k \ne 2$, then both sequences split and encode the direct sum decomposition: $V \ot V \cong \Lambda^2(V) \oplus S^2(V)$.

\subsection{Representations}
If $V$ is a representation of a group $G$, then so are the vector spaces deduced from $V$ such as $S''_2(V^*)$.  

If $\car k \ne 2$, then replacing $V$ with $V^*$ in \eqref{wedgeS.seq} gives a direct sum decomposition $(\Bil V)^G \cong (\Lambda^2 V^*)^G \oplus (\Quad V)^G$. Furthermore, if $(\Bil V)^G = kb$ for some nonzero $b$, then $b$ is either symmetric or alternating.  

However, when $\car k = 2$ the situation is slightly different.

\begin{lemma}  \label{char2.sym}
Suppose $\car k = 2$ and $V$ is a representation of a group $G$ such that $(\Bil V)^G = kb$ for some nonzero $b$.  Then $b$ is symmetric and the quadratic form $\psi(b)$ is totally singular.
\end{lemma}

\begin{proof}
The exact sequence \eqref{bil.seq} yields an exact sequence
\[
\begin{CD}
0 @>>> S'_2(V^*)^G @>>> (\Bil V)^G @>\alt>> \Lambda^2(V^*)^G,
\end{CD}
\]
and $\alt(b) = mb$ for some $m \in k$.  Writing down a Gram matrix $M$ for $b$ with respect to some basis of $V$, the equation $\alt(b) = mb$ says that $M - M^t = mM$, hence $m = 0$, i.e., $b$ is symmetric.  

For the second claim, suppose $\varphi \psi(b) \ne 0$.  Then the alternating form $\varphi \psi(b)$ equals $nb$ for some $n \in \kx$, thus $b$ is alternating and $\psi(b) = 0$. This is a contradiction, hence $\varphi \psi(b) = 0$.
\end{proof}

\subsection{Symmetric $p$-linear forms versus symmetric powers of degree $p$} \label{symvsym.2}

We now return to the notions introduced in this section in the special case where $k$ has characteristic $p \ne 0$, and 
compare the symmetrized $p$-linear tensors, $S''_p(V)$ and the symmetric powers of degree $p$, $S^p(V)$.  

Consider, for example, the case where $V$ has basis $v_1, v_2$ and $p = 3$.  Then  $S''_3(V)$ is 2-dimensional with basis $v_i \ot v_i \ot v_{3-i} + v_i \ot v_{3-i} \ot v_i + v_{3-i} \ot v_i \ot v_i = \varphi(v_i^2 v_{3-i})/2$ for $i = 1, 2$.  This contradicts the (erroneous) formula for $\dim S''_p(V)$ given in Exercise 5b from Chapter III, \S6 of \cite{Bou:alg1}; a correct formula is implicit in the following result.

\begin{prop} \label{symvsym.prop}
Suppose $\car k = p$.  Then the sequence
\begin{equation} \label{symvsym.seq}
\begin{CD}
0 @>>> V^{[1]} @>>> S^p(V) @>\varphi>> S''_p(V) @>>> 0
\end{CD}
\end{equation}
is exact, $V^{[1]} = \psi(S'_p(V))$, and $S'_p(V)$ is a direct sum of $S''_p(V)$ and the $k$-span of $\{ \ot^p v \mid v \in V \}$.
\end{prop}

\begin{proof} For exactness, the only thing to check is exactness at the middle term $S^p(V)$.  
Fix a basis $v_1, \ldots, v_d$ of $V$ and form the corresponding basis $\B= \{ v_1^{e_1} v_2^{e_2} \cdots v_d^{e_d} \mid \sum e_i = p \}$ of $S^p(V)$. We 
partition $\B$ as $X = \{ v_i^p \mid 1 \le i \le d \}$ and $Y = \B \setminus X$.  We may identify $V^{[1]}$ with the $k$-span of $X$, because, for any $v = \sum c_i v_i \in V$, we have $v^p = \sum c_i^p v_i^p$.

For an element $v_1^{e_1} v_2^{e_2} \cdots v_d^{e_d}$, we define
\[
h := \left( \ot^{e_1} v_1 \right) \ot \left( \ot^{e_2} v_2 \right) \ot \cdots \ot \left( \ot^{e_d} v_d \right) \quad \in \ot^p V
\]
and $H$ to be the stabilizer in $\Sigma_p$ of $h$.  Clearly,
\[
\varphi(v_1^{e_1} v_2^{e_2} \cdots v_d^{e_d}) = |H| \cdot \sum_{\sigma \in \Sigma_p/H} \s h.
\]
Therefore, for $v_i^p \in X$, we have $H = \Sigma_p$, so $\varphi(v_i^p) = p! \cdot h = 0$.  For $h = v_1^{e_1} \cdots v_d^{e_d} \in Y$, the size of the coset space $\Sigma_p / H$ is $p! / (e_1! e_2! \cdots e_d!)$, which is divisible by $p$ because no $e_i$ is.  Hence, $p$ does not divide $|H|$ and $\varphi(h)$ is in $k^\times \cdot \sum_{\sigma \in \Sigma_p/H} \s h$.  For distinct cosets $\s H$ and $\s' H$ of $H$, the elements $\s H \cdot h$ and  $\s' H \cdot h$ are linearly independent in $\ot^p V$, so it follows that $\ker \varphi = \Span(X) = V^{[1]}$.  

From this argument, the final claim is clear.
\end{proof}

Replacing $V$ with $V^*$, we may view sequence \eqref{symvsym.seq} as relating homogeneous polynomials of degree $p$ and symmetric $p$-linear forms.

\begin{defn}
For $k$ a field of characteristic $p$, a symmetric $p$-linear form $f$ is \emph{characteristic} if $f(v, v, \ldots, v) = 0$ for all $v \in V$.  For $p = 2$, a characteristic symmetric  2-linear form is nothing but an alternating bilinear form.
\end{defn}

\begin{cor} \label{symvsym.char}
For every vector space $V$ over a field of characteristic $p$, $S''_p(V^*)$ is the space of symmetric $p$-linear forms on $V$ that are characteristic.
\end{cor}

\begin{proof}
For $f \in S'_p(V^*)$, $f$ is not characteristic iff $f(v_1,v_1,\ldots, v_1) \ne 0$ for some $v_1 \in V$ iff with respect to some basis $v_1, \ldots, v_n$ of $V$ with dual basis $x_1, \ldots, x_n$, when we write $f$ in terms of the dual basis, we find a nonzero multiple of $\ot^p x_1$.  Proposition \ref{symvsym.prop} gives the claim.
\end{proof}

For a different view on Proposition \ref{symvsym.prop} and Corollary \ref{symvsym.char}, see \cite[Section 3, esp.~Theorem~3.4]{DrapalV}.

\begin{cor} \label{symvsym}
Suppose $\car k = p$ and $V$ is a representation of an algebraic group $G$.
\begin{enumerate}
\item \label{symvsym.inj} If $V$ has no codimension-1 G-submodules, then $\varphi \!: S^p(V^*)^G \to S''_p(V^*)^G$ is injective.
\item \label{symvsym.surj} If $\operatorname{H}^1(G, V^{*[1]}) = 0$, then $\varphi \!: S^p(V^*)^G \to S''_p(V^*)^G$ is surjective.
\end{enumerate}
\end{cor}

Part \eqref{symvsym.surj}, in the special case $p = 2$, can be found in \cite[Satz 2.5]{Willems77}.

\begin{proof} 
Taking the exact sequence  \eqref{symvsym.seq}, replacing $V$ with $V^*$, and taking fixed submodules  gives the exact sequence
\[
\begin{CD}
(V^{*[1]})^G @>>> S^p(V^*)^G @>\varphi>> S''_p(V^*)^G @>>> \H^1(G, V^{*[1]}).
\end{CD}
\]
From this, \eqref{symvsym.surj} is clear.  For \eqref{symvsym.inj}, if $(V^*)^G = 0$, then $(V^{*[1]})^G = 0$.
\end{proof}

\section{Bilinear Forms on Irreducible and Weyl Modules} \label{bil.weyl}

\subsection{Bilinear forms on irreducible modules} The proof in the case $k = \C$ given in \cite[Th.~3.2.13, 3.2.14]{GW2} shows that, for every field $k$ and every $\la \in X(T)_+$, \emph{$(\Bil L(\la))^G$ is nonzero if and only if $\la = -w_0 \la$, if and only if $(\Bil L(\la))^G = kb$ for some nondegenerate $b$.}

Suppose these conditions hold.  If $\car k \ne 2$, then the splitting of sequence \eqref{wedgeS.seq}, shows that $b$ is symmetric or skew-symmetric.  \emph{If $\car k = 2$ and $\la \ne 0$, then $b$ is alternating.}  Indeed, in that case $b$ is symmetric with $\psi(b)$ totally singular by Lemma \ref{char2.sym}, but $\psi(b)$ is the zero quadratic form by Corollary \ref{symvsym}\eqref{symvsym.inj}, i.e., $b$ is alternating, as claimed.

\subsection{Reducible modules}
The material in the preceding subsection is enough to determine $(\Bil V)^G$ when $V$ is semisimple.  We now consider arbitrary (finite-dimensional) representations $V$.  Recall that the \emph{socle}, denoted by $\soc V$, is the largest semisimple submodule of $V$.  
The \emph{head} of $V$, denoted by $\head V$, is the maximal semisimple quotient; the kernel of the map $V \to \head V$ is the \emph{radical} of $V$, denoted by $\rrad V$.
The following observation has many applications.

\begin{lem} \label{headvsocle}
Let $U \subseteq \rrad V$.  If $U$ and $(\head V)^*$ have no common composition factors, then the pullback $(\Bil V/U)^G \to (\Bil V)^G$ is an isomorphism.
\end{lem}

\begin{proof}
Suppose first that $U$ is simple.  The induced map $\Bil(V/U) \to \Bil V$ is obviously injective and $G$-equivariant.  For every $\bhat \in \Hom_G(V, V^*)$, we have
\[
\bhat(U) \subseteq \socle(V^*) = (\head V)^*,
\]
so by hypothesis $\bhat$ vanishes on $U$.  Furthermore, $\bhat$ induces a homomorphism $\head V \to \head (V^*) \to U^*$, which must also vanish by the hypothesis, hence every element of $\bhat(V)$ vanishes on $U$, and $\bhat$ is in the image of $\Hom_G(V/U, (V/U)^*) \to \Hom_G(V, V^*)$, as claimed.

The general case follows by induction on the length of a composition series for $U$, because for simple $U_0 \subseteq \rrad V$ we have $\rrad(V)/U_0 = \rrad(V/U_0)$ and $\head(V/U_0) = \head(V)$.
\end{proof}

\subsection{Weyl modules} For each $\la \in X(T)_+$,  the Weyl module $V(\la)$ has head $L(\la)$ and therefore $\rrad V(\la)$ is the kernel of the map $V(\la) \to L(\la)$ from \eqref{triple}.  Note that $\rrad V(\la) = 0$ if and only if $V(\la)$ is irreducible, and this is true if and only if all maps in \eqref{triple} are isomorphisms.
If $\car k = 0$ or if $\la$ is minuscule, then these equivalent conditions hold.

\begin{example} \label{qf.eg}
Let $G$ be an adjoint group of type $B_n$ over a field $k$, hence it is $\SO(V,q)$ for a quadratic form $q$ on a $(2n+1)$-dimensional vector space $V$ by \cite[p.~364]{KMRT}. The $G$-module 
$V$ can be identified with the Weyl module $V(\omega_1)$.  If $\car k = 2$, the radical of the bilinear form $\rrad b_q$ (whose definition is recalled in Example \ref{bil.eg}) is 1-dimensional, the irreducible representation 
$L(\omega_1)$ is $V(\omega_1)/\rrad b_q$, and $H^0(\la) = V^*$ has a unique proper submodule, $L(\omega_1)$.
\end{example}


\begin{lem} \label{bil}
For $\la \in X(T)_+$, the surjection $V(\la) \to L(\la)$ induces an isomorphism $(\Bil L(\la))^G \iso (\Bil V(\la))^G$.
\end{lem}

\begin{proof}[Proof \#1]
$(\Bil V(\la))^G = \Hom_G(V(\la), H^0(-w_0\la))$, which is $k$ (iff $\la = -w_0 \la)$ or $0$.  Thus the induced map $(\Bil L(\la))^G \to (\Bil V(\la))^G$ is onto.
\end{proof}

\begin{proof}[Proof \#2]
Apply Lemma \ref{headvsocle} with $U = \rrad V(\la)$, using Lemma \ref{comparable} to see that $L(-w_0\la)$ is not a composition factor of $U$.
\end{proof}

%
%
%
\begin{defn}
We say a representation $V$ of $G$ is \emph{symplectic} if there is a nonzero $G$-invariant alternating bilinear form on $V$ (i.e., $\Lambda^2(V^*)^G \ne 0$) and $V$ is \emph{orthogonal} if $(\Quad V)^G \ne 0$.
Clearly, if $\car k \ne 2$, a representation $V(\la)$ or $L(\la)$ can be symplectic or orthogonal or neither, but not both.  If $\car k = 2$ and $\la \ne 0$, it can be symplectic, both orthogonal and symplectic, or neither. 
\end{defn}

\subsection{Integral models}  \label{integral}
As $G$ is a split reductive group, there is a split reductive group scheme $G_\Z$ over $\Z$ such that $G_\Z \times k$ is isomorphic to $G$ \cite{SGA3:new}.  Moreover, for each $\la \in X(T)_+$ there is a representation $V(\la, \Z)$ of $G_\Z$ such that base change identifies $V(\la, \Z) \times k$ with the Weyl module $V(\la)$ of $G$ over $k$.  Consequently, it makes sense to write $V(\la, K)$ for the Weyl module $V(\la, \Z) \times K$ of $G_\Z \times K$, for any field $K$. We use this convention when we want to emphasize the field of definition.

Suppose that $\la = -w_0 \la$, so $V(\la, \C)$ is orthogonal or symplectic; the recipe described in \cite{GW2} will tell which.  Because $V(\la, \Q)$ is also self-dual, there is a nonzero $(G_\Z \times \Q)$-invariant bilinear form on $V(\la, \Q)$, and by clearing denominators we may assume it is indivisible and defined on $V(\la, \Z)$ with values in $\Z$.  From this we find: when $\car k \ne 2$, $V(\la, k)$ is orthogonal (resp., symplectic) iff $V(\la, \C)$ is orthogonal (resp., symplectic).  

For any $k$, if  $V(\la, \C)$ is orthogonal, then we can similarly use the symmetric bilinear form on $V(\la, \Q)$ to construct a $(G_\Z \times \Q)$-invariant quadratic form that is nonzero and indivisible on $V(\la, \Z)$ and so conclude that $V(\la, k)$ is orthogonal.  The converse of this is false, see for example Proposition \ref{Sp.odd}.

Entirely parallel remarks hold for the induced module, $H^0(\la)$.



\subsection{Reduced Killing Form} \label{killing.2}
Let $G$ be a split \emph{quasi-simple} group defined over $\Z$.
the highest root $\at$ is in $X(T)_+$ and the Weyl module $V(\at, \Z)$ is the Lie algebra $\gt$ of
the simply connected cover $\Gt$ of $G$ \cite[2.5]{G:vanish}.
Dividing the Killing form of $\gt$ by twice the dual Coxeter number $\hch$ gives an even and indivisible symmetric bilinear form --- cf. \ \cite[p.~633]{GrossNebe} or \cite[pp.~180, 181]{SpSt} --- so there exists a unique indivisible quadratic form $s$ so that $2\hch \cdot b_s$ is the Killing form $\kappa$.  This 
is called the \emph{reduced Killing quadratic form}.

We now sketch how to determine the isomorphism class of $s$ over any field $k$.  It is harmless to assume that $G$ is simply connected.  The roots $\Phi$ and simple roots $\Pi$ index a basis $\{ h_\delta \mid \delta \in \Pi \} \cup \{ x_\alpha, x_{-\alpha} \mid \alpha \in \Phi \}$ for $\Lie(G)$.
As $s$ is invariant under $T$, it vanishes on each $x_{\pm \alpha}$, and it quickly follows that $s$ is an orthogonal sum of its restrictions to $\Lie(T)$ and $\Z x_\alpha + \Z x_{-\alpha}$ for each $\alpha \in \Phi$.  Put $r$ for the square-length ratio of long roots to short roots, so $r \in \{ 1, 2, 3 \}$.  The calculations in \cite[pp.~180, 181]{SpSt} show that $\Z x_\alpha + \Z x_{-\alpha}$ contributes a zero form to $s \ot k$ if $\alpha$ is short and $r$ is zero in $k$, otherwise it contributes a hyperbolic plane.  

As for the restriction to $\Lie(T)$, recall that there is a unique positive-definite quadratic form $q^\vee$ on the coroot lattice $Q^\vee$ (for the simple root system $\Phi$ of $G$) that takes the value 1 on short coroots and $r$ on long coroots.  Since $s$ restricts to a Weyl-invariant form on $\Lie(T)$, the formulas in \cite{SpSt} show that the restriction of $s$ to $\Lie(T)$ is $q^\vee$.  In summary, $s \ot k$ is an orthogonal sum of hyperbolic planes, a zero form (if $\car k \mid r$) and $q^\vee \ot k$.

To calculate $q^\vee \ot k$, fix a basis $\alpha_1^\vee, \ldots, \alpha_n^\vee$ of simple coroots and set $C$ to be the Cartan matrix with respect to the basis $\alpha_1, \ldots, \alpha_n$ of simple roots.  For $D$ the diagonal matrix whose $i$-th diagonal entry is the square-length $q^\vee(\alpha_i^\vee)$ of $\alpha_i^\vee$, the product $DC$ is a symmetric integer matrix with even entries on the diagonal, and 
\begin{equation} \label{q.def}
\text{$q^\vee(v) = \frac12 v^T DC v$ for $v \in Q^\vee$}.
\end{equation}
In case $\car k = 2$, formulas for the isometry class of $s$ can be found in \cite[\S3]{BabicCh}.

\subsection{A uniserial example} \label{uniserial.eg}
Let $\la, \mu \in X(T)_+$ such that $\mu = -w_0 \mu$, and let $V$ be a uniserial $G$-module with composition factors
\[
L(\mu), L(\la), L(\mu).
\]
Note that $V$ does not satisfy the hypotheses of Lemma \ref{headvsocle}.

\begin{eg}
\emph{If $\la = -w_0 \la$ and $\Ext^1_G(L(\la), L(\mu)) = k$, then $\dim(\Bil V)^G = 2$.}  To see why this is so, note that $V^*$ is also uniserial with the same composition factors, and that the corresponding module diagrams in the sense of \cite{BensonCarlson} are ``rigid'' by the hypothesis on $\Ext$, so by Proposition 6.5 of ibid. we can read off the elements of $\Hom_G(V, V^*)$ from the diagrams---clearly $\Hom_G(V, V^*)$ is 2-dimensional with a basis consisting of an isomorphism and a map that sends $V$ onto $\socle(V^*)$.
\end{eg}

The following provides a tool to check the $\Ext$ hypothesis in the example.

\begin{lem}\label{extlem} If $V(\la_2)$ has two composition factors $L(\lambda_{2})$ and $L(\la_1)$ (with $L(\lambda_{1})$ as the socle) and $H^{0}(\la_1) = L(\la_1)$, then $\Ext^1_G(L(\la_1), L(\la_2)) \cong k$.
\end{lem}

\begin{proof}
Apply $\Hom_G(?, L(\la_1))$ to the exact sequence
\[
\begin{CD}
0 @>>> L(\la_1) @>>> V(\la_2) @>>> L(\la_2) @>>> 0
\end{CD}
\]
to get
\begin{multline*}
\Hom_G(V(\la_2), L(\la_1)) \to \Hom_G(L(\la_1), L(\la_1)) \to \\ \to\Ext^1_G(L(\la_2), L(\la_1)) \to \Ext^1_G(V(\la_2), L(\la_1)).
\end{multline*}
The first term is zero because $\la_1 \ne \la_2$.  The last term is zero by \cite[Prop.~II.4.16]{Jantzen} because  $L(\la_1)$
has a good filtration.
\end{proof}

\section{Bilinear Forms on Induced Modules}

\subsection{ } The theory of bilinear forms on an induced module $H^0(\la)$ is notably different from that for irreducible and Weyl modules, and in general the forms on $H^0(\la)$ need not have much to do with $L(\la)$.  We start with the  following basic lemma.

\begin{lem} \label{comparable}
If $\la \in P_+$ and $-w_0\la$ are comparable in the partial ordering on $P$, then $\la = -w_0\la$.
\end{lem}

\begin{proof}
$-w_0 \la = \la + \s$ for $\s$ a sum of positive roots or a sum of negative roots.
Applying $-w_0$ to both sides and subtracting, we find that $\s = w_0 \s$, but $w_0$ interchanges positive and negative roots, hence $\s = 0$, i.e., $\la = -w_0 \la$.  
\end{proof}

Then we find:

\begin{lemma} \label{induced.lem}
If $H^0(\la)$ is reducible, then the pullback 
$$(\Bil H^0(\la)/L(\la))^G \to (\Bil H^0(\la))^G$$ 
is an isomorphism.
\end{lemma}

\begin{proof}
The dual of $\head H^0(\la)$ is the socle of $H^0(\la)^* = V(-w_0\la)$.  By Lemma \ref{comparable}, $\la$ cannot be less than $-w_0\la$, and therefore $L(\la)$ cannot be a component of the socle.  The conclusion follows by Lemma \ref{headvsocle}.
\end{proof}


\begin{eg}
Let $G$ be a quasi-simple group.  Then
\[
\dim (\Bil H^0(\at))^G = \begin{cases}
4 & \text{if $\car k = 2$ and $G$ has type $D_n$ for $n \ge 4$ and even;} \\
2 & \text{if $\car k = 2$ and $G$ has type $B_n$ or $C_n$ with $n \ge 2$;} \\
1 & \text{otherwise.}
\end{cases}
\]
Note that the dimension can be bigger than 1, unlike for irreducible and Weyl modules.
To see the claim, we combine Lemma \ref{induced.lem} with the preceding discussion and with the $G$-module structure of $V(\at) = H^0(\at)^*$ given in \cite{Hiss}.  Put $\hsr$ for the highest short root.

Suppose $G$ has type $B_n$ or $C_n$ for $n \ge 2$ and $\car k = 2$.  Then $H^0(\at)/L(\at)$ is either $k \oplus L(\hsr)$ or is uniserial with composition factors $k$, $L(\hsr)$, $k$ as in Section \ref{uniserial.eg}.  In the latter case, $V(\hsr)$ has socle $k$ by \cite{Hiss}, so Lemma \ref{extlem} applies and in both cases we find $\dim (\Bil H^0(\at))^G = 2$ as claimed.

If $G$ has type $F_4$ and $\car k = 2$, or if $G$ has type $G_2$ and $\car k = 3$, then $H^0(\at)/L(\at)$ is $L(\hsr)$, so $\dim (\Bil H^0(\at))^G = 1$.

In the remaining cases, writing $Z$ for the scheme-theoretic center of the simply connected cover of $G$, $H^0(\at)/L(\at) \cong \Lie(Z)^*$, on which $G$ acts trivially.  If $Z$ is not \'etale, then $\dim (\Bil H^0(\at))^G = (\dim \Lie(Z))^2$, whence the claim.

Note that, in calculating the dimension of $(\Bil H^0(\at))^G$, we implicitly gave formulas for all of the $G$-invariant bilinear forms.

We remark that for $G$ of type $A$, $D$, or $E$, $H^0(\at)$ is the Lie algebra of the adjoint group \cite[3.5(2)]{G:vanish}.
\end{eg}


\subsection{A necessary condition}

\begin{lem} \label{C:inducedbilinear}
If there is a nonzero $G$-invariant bilinear form on $H^0(\la)$ or $T(\la)$, then $2\la$ is in the root lattice $Q$.
\end{lem}

\begin{proof}
On the one hand, the action of the torus $T$ on the representation $V = H^0(\la)$ or $T(\la)$ turns $V^* \ot V^* = \Bil(V)$ into a graded vector space with grade group $X(T)$, and the hypothesis gives that 0 is a weight.  On the other hand, all weights of $V$ are congruent to $\la$ mod the root lattice $Q$, hence all weights of $\Bil(V)$ are congruent to $-2w_0\la$ mod $Q$, so $-2w_0\la \in Q$.  As $-w_0$ normalizes $Q$, the conclusion follows.
\end{proof}

Note that $-w_0$ acts on $P/Q$ as $-1$, hence the condition $\la = -w_0 \la$ (for the existence of a nonzero $G$-invariant bilinear form on $V(\la)$ or $L(\la)$) implies that $2\la \in Q$.

\subsection{} The following result demonstrates that the number of $G$-invariant bilinear forms on the induced module 
$H^{0}(\lambda)$ is related to the rational $B$-cohomology. In \cite{HemmerNakano}, it was shown that similar $B$-cohomology calculations 
for $\GL_{n}(k)$ are related to the cohomology of Specht modules for the symmetric group on $n$ letters. 

\begin{theorem} \label{H0.bil}
Let $\lambda\in X(T)_{+}$. Then for $N = |\roots|/2$,
\begin{align*}
\Nr{G}{H^{0}(\lambda)}&=\dim \operatorname{Ext}^{N}_{B}(H^{0}(\lambda),-\lambda-2\rho) \\
&=
\dim \operatorname{H}^{N}(B,V(-w_{0}\lambda)\otimes (-\lambda-2\rho)).
\end{align*}
\end{theorem} 

\begin{proof} Recall that $H^{0}(\lambda)^{*}=V(-w_{0}\lambda)$. Furthermore, by using 
Serre duality \cite[II 4.2(9)]{Jantzen}, 
$$V(-w_{0}\lambda)\cong H^{N}(w_{0}\cdot (-w_{0}\lambda))\cong H^{N}(-\lambda-2\rho).$$ 

Consider the following spectral sequence \cite[I.4.5]{Jantzen}
$$
E_{2}^{i,j}=\Ext^{i}_{G}(H^{0}(\lambda),R^{j}\text{ind}_{B}^{G}\ (w_{0}\cdot (-w_{0}\lambda)))\Rightarrow 
\Ext^{i+j}_{B}(H^{0}(\lambda),w_{0}\cdot (-w_{0}\lambda)).$$
According to Serre duality \cite[II.4.2(9)]{Jantzen}, 
\begin{eqnarray*} 
R^{i}\text{ind}_{B}^{G}\ (w_{0}\cdot (-w_{0}\lambda))
&\cong& [R^{N-i}\text{ind}_{B}^{G}\ (-(w_{0}\cdot (-w_{0}\lambda)+2\rho))]^*\\
&\cong& [R^{N-i}\text{ind}_{B}^{G}\ (\lambda)]^*.
\end{eqnarray*} 
By assumption $\lambda\in X(T)_{+}$, so by Kempf's vanishing theorem, 
$$R^{N-i}\text{ind}_{B}^{G}\ (\lambda)=0$$ 
when $N-i>0$ (or $N>i$). 

This shows that there is only one non-zero row in the 
spectral sequence, thus the spectral sequence thus collapses and for all $i>0$: 
\begin{align*}
\Ext^{i}_{G}(H^{0}(\lambda),R^{N}\text{ind}_{B}^{G}\ (w_{0}\cdot (-w_{0}\lambda)))&\cong
\Ext^{i+N}_{B}(H^{0}(\lambda),w_{0}\cdot (-w_{0}\lambda)) \\
&\cong \Ext^{i+N}_{B}(H^{0}(\lambda),\lambda-2\rho) 
\end{align*} 
The result now follows by specializing to the case when $i=0$. 
\end{proof}

%
%

\subsection{$\textbf{SL}_n$ examples; symmetric powers} 
Let $G=\SL_{n}$ and consider $H^{0}(d\omega_{1})$ where $\omega_{1}$ is the first fundamental weight. 
We have $H^{0}(d\omega_{1})\cong S^{d}(V)$ where $V$ is the natural $n$-dimensional representation. 

\begin{example} For $G=\SL_{2}$, $H^{0}(d\omega_{1})^{*}\cong V(d\omega_{1})$ has a simple socle \cite[II.5.16 Corollary]{Jantzen} and 
$H^{0}(d\omega_{1})$ is multiplicity-free as a $G$-module (note the weight spaces are all one-dimensional).  So taking $U = \rrad H^0(d\omega_1)$ in Lemma \ref{headvsocle} gives 
\[
(\Bil H^0(d \omega_1))^G = (\Bil \head H^0(d \omega_1))^G \cong k.
\]
That is, base change from the integral model as in Section \ref{integral} provides a nonzero $\SL_2$-invariant bilinear form on $S^d(k^2)$, and it is the only one up to a factor in $k^\times$.
\end{example}

\begin{example} Let $G=\SL_{n}$ where $n\geq 3$. The $G$-modules $H^{0}(d\omega_{1})$ can have complicated submodule structures. However, 
one can employ the weight criterion in Corollary~\ref{C:inducedbilinear} to deduce the following: If $n\nmid 2d$ then $2d\omega_{1}\notin {\mathbb Z}\roots$, and 
$\Nr{G}{H^{0}(d\omega_{1})}=0$. 
\end{example} 

\subsection{$\SL_3$ examples}  \label{SL3.H}

\begin{example} Let $G=\SL_{3}$ and $\lambda$ be a generic weight in the lowest $p^{2}$-alcove (see \cite{DotySullivan}). Then the following are true: 
(i) the composition factors of $H^{0}(\lambda)$ are multiplicity free and (ii) the head of $H^{0}(\lambda)$ is a single irreducible representation. 
Moreover, if $\lambda\neq -w_{0}\lambda$ then the head of $H^{0}(\lambda)$ is not a composition factor of $V(-w_{0}\lambda)$. 
Therefore, in this case, $\Nr{G}{H^{0}(\lambda)}=0$.  
\end{example}

Suppose $p = \car k \ge 3$ and consider the following 
weights of $G = \SL_3$ expressed in terms of the fundamental weights: 
\begin{eqnarray*} 
\lambda_{1}&=& (0,0) \\
\lambda_{2}&=& s_{\alpha_{1}+\alpha_{2},p(\alpha_{1}+\alpha_{2})}\cdot \lambda_{1} =(p-2,p-2) \\
\lambda_{3}&=& s_{\alpha_{1},p\alpha_{1}}\cdot \lambda_{2}=(p,p-3) \\
\lambda_{4}&=& s_{\alpha_{2},p\alpha_{2}}\cdot \lambda_{2}=(p-3,p)  = -w_0 \la_3
\end{eqnarray*} 
We have indicated how $\lambda_{j}$, for $j=1,2,3$ are linked to $(0,0)$ under the dot action of the affine Weyl group $W_{p}$.

By using the standard translation functor arguments \cite[II.7.19, II.7.20]{Jantzen} or employing \cite{DotySullivan} we can deduce the following facts.  
The representation $H^{0}(\lambda_{2})$ is uniserial with two composition factors (from head to socle): 
\begin{equation} \label{SL3.la2}
L(0,0),\ L(\lambda_{2})
\end{equation}
For $j = 3, 4$, $H^{0}(\lambda_{j})$  is uniserial with two composition factors (from head to socle): 
$$L(\lambda_{2}),\ L(\lambda_{j}).$$ 

\begin{example}  For $j = 2, 3, 4$, $(\Bil H^{0}(\lambda_j))^{G} = (\Bil \head H^0(\lambda_j))^G \cong k$ by Lemma \ref{headvsocle}.
Pulling back along the surjection $T(\la_3) \to H^0(\la_3)$ from \eqref{triple} gives a nonzero $G$-invariant bilinear form on the tilting module $T(\la_3)$ also.  (In fact it generates $(\Bil T(\la_3))^G$ by Example \ref{tilting.eg} below.)
That is, $H^0(\la_3)$ and $T(\la_3)$ each have a nonzero $G$-invariant bilinear form, yet $\la_3 \ne -w_0 \la_3$, in contrast with the situation for simple and Weyl modules described in Section \ref{bil.weyl}.
\end{example} 

\section{Bilinear Forms on Tilting Modules}

\subsection{}  The tilting module $T(\lambda)$ has both a good and Weyl filtration, and the composition factor of highest weight in $T(\lambda)$ 
is $L(\lambda)$.  We briefly discuss the maps induced by applying the functor $V \mapsto (\Bil V)^G$ to \eqref{triple}.

\begin{lem} \label{T.pull}
For all $\la \in X(T)_+$, the pullback map $(\Bil T(\la))^G \to (\Bil V(\la))^G$ is surjective.  If $T(\la)$ is reducible, then the composition 
\[
(\Bil H^0(\la))^G \to (\Bil T(\la))^G \to (\Bil V(\la))^G
\]
is zero.
\end{lem}

\begin{proof}
Assume $(\Bil V(\la))^G \ne 0$, so $\la = -w_0 \la$ and there exists an isomorphism $f \!: T(\la) \xrightarrow{\sim} T(-w_0 \la) = T(\la)^*$.  As $\la$ and $-w_0\la$ are weights of $T(\la)$ (and $V(\la)$) of multiplicity 1, the pullback of $f$ to $V(\la)$ is nonzero, proving the first claim.  Commutativity of \eqref{triple} and Lemma \ref{induced.lem} give the second claim.
\end{proof}

The lemma shows that when $T(\la)$ is reducible:
\begin{equation} \label{T.red}
\dim (\Bil T(\la))^G \ge \dim (\Bil H^0(\la))^G + \dim (\Bil V(\la))^G.
\end{equation}

\subsection{} We now compute the dimension of the space $G$-invariant bilinear forms on $T(\lambda)$ in terms of the good filtration multiplicities. 
Define $[T(\lambda):H^{0}(\mu)]$ to be the number of times $H^{0}(\mu)$ appears in a good filtration for $T(\lambda)$; it equals $\Hom_G(V(\mu), T(\la))$ by \cite[Prop.~II.4.16(a)]{Jantzen}, so it
is independent of the choice of good filtration.

\begin{theorem}\label{tiltingforms} Let $G$ be a reductive algebraic group. Then for 
each $\lambda\in X(T)_{+}$,  
\begin{enumerate} 
\item \label{tilting.gen} $\Nr{G}{T(\lambda)}=\sum_{\mu\in X(T)_{+}}[T(-w_{0}\lambda):H^{0}(\mu)][T(\lambda):H^{0}(\mu)]$. 
\item \label{tilting.part} If $\lambda=-w_{0}\lambda$ then 
$\Nr{G}{T(\lambda)}=\sum_{\mu\in X(T)_{+}} [T(\lambda):H^{0}(\mu)]^{2}$.  
\end{enumerate} 
\end{theorem}

\begin{proof} \eqref{tilting.gen} According to \cite[II 4.13 Proposition]{Jantzen}, $\Ext_{G}^{1}(V(\sigma_{1}),H^{0}(\sigma_{2}))=0$ for all $\sigma_{1},\sigma_{2}\in X(T)_{+}$.
Since $T(\lambda)$ has a Weyl filtration, it follows that $\Ext_{G}^{1}(T(\lambda),H^{0}(\sigma))=0$ for all $\lambda,\sigma\in X(T)_{+}$.
Therefore, the functor $\Hom_{G}(T(\lambda),-)$ is exact on short exact sequences of modules that admit good filtrations. 
Now $T(-w_{0}\lambda)$ admits a good filtration, thus 
\begin{align*}
\Nr{G}{T(\lambda)} &= \dim \Hom_{G}(T(\lambda),T(-w_{0}\lambda)) \\
&=\sum_{\mu\in X(T)_{+}} [T(-w_{0}\lambda):H^{0}(\mu)]\dim \Hom_{G}(T(\lambda),H^{0}(\mu)).
\end{align*}
Because $T(\lambda)=T(\lambda)^{\tau}$ and 
$H^{0}(\lambda)=V(\lambda)^{\tau}$ under the duality $\tau$ defined in \cite[II 2.12, 2.13]{Jantzen}, $\Hom_G(T(\la), H^0(\mu)) = \Hom_G(V(\mu), T(\la))$ and \eqref{tilting.gen} follows.
Part \eqref{tilting.part} follows immediately from \eqref{tilting.gen}. 
\end{proof} 

In the sums in Theorem \ref{tiltingforms}, the $\mu = \la$ term contributes 0 if $\la \ne -w_0 \la$ (by Lemma \ref{comparable}) and 1 if $\la = -w_0 \la$.  In either case, the sum restricted to $\mu \ne \la$ gives the dimension of the kernel of the pullback $(\Bil T(\la))^G \to (\Bil V(\la))^G$.
 
\subsection{$\SL_3$ examples} Let $G=\SL_{3}$ with $p\geq 3$ and let $\la_j$ for $j = 1, 2, 3, 4$ be as in Section \ref{SL3.H}. 

\begin{example}  \label{tilting.la2}
The tilting module $T(\lambda_{2})$ is uniserial with composition factors (from the head to the socle):
$$L(0,0),\  L(\la_2),\  L(0,0).$$
As in Section \ref{uniserial.eg}, we find that $\dim (\Bil T(\la_2))^G = 2$, which agrees with the $1^2 + 1^2 = 2$ provided by
Theorem~\ref{tiltingforms}\eqref{tilting.part}, and we find equality in  \eqref{T.red}.  
\end{example} 

\begin{example} \label{tilting.eg} The structure of the tilting module $T(\la_j)$ for $j = 3, 4$ is given by the following diagrams in the style of \cite{BensonCarlson}, with the head is on top and the socle is on the bottom:
\[
\xymatrix@C=0.1pc@R=0.5pc{
&&L(\la_2)& \\
&L(\la_j) \ar@{-}[ur] \ar@{-}[dr] & & L(0,0) \ar@{-}[ul] \ar@{-}[dl] \\
&& L(\la_2) & }
\]
Therefore, $\Nr{G}{T(\lambda_{j})}=1$ which agrees with Theorem~\ref{tiltingforms}\eqref{tilting.gen} by looking at the structure of the tilting module above with its 
good filtration factors $H^0(\la_j)$, $H^0(\la_2)$.
\end{example} 

\section{Quadratic Forms on Tilting modules}

\subsection{Modules with a good filtration} We characterize symmetric $G$-invariant $p$-linear forms for  tilting modules and Weyl modules via cohomological 
vanishing. 

\begin{prop} \label{TV} 
Let $G$ be a simple split algebraic group and let $V$ be a finite-dimensional $G$-module such that $V^{*}$ admits a good filtration. Assume further that 
if $p=2$ and $G$ is of type $C$ then $[V^{*}:H^{0}(\omega_{1})]=0$. Then every $G$-invariant characteristic symmetric $p$-linear form on $V$ is the polarization of a $G$-invariant 
homogeneous polynomial of degree $p$ on $V$. 
\end{prop} 

\begin{proof} By Corollary \ref{symvsym}\eqref{symvsym.surj}, it suffices to prove that $E_{1}:=\H^{1}(G,(V^{*})^{[1]})=0$. Apply the Lyndon-Hochschild-Serre spectral 
sequence: 
\begin{equation}
E_{2}^{i,j}=\H^{i}(G/G_{1},\H^{j}(G_{1},k)\otimes (V^{*})^{[1]})\Rightarrow \H^{i+j}(G,(V^{*})^{[1]}).
\end{equation}
with the five term exact sequence $0\rightarrow E_{2}^{1,0}\rightarrow E_{1} \rightarrow E_{2}^{0,1} \rightarrow  E_{2}^{2,0}\rightarrow E_{2}$. 

Since $V^{*}$ admits a good filtration, it follows that $E_{2}^{i,0}=0$ for $i\geq 1$. Therefore, 
\begin{align*}
\H^{1}(G,(V^{*})^{[1]})&\cong \Hom_{G/G_{1}}(k,\Ext^{1}_{G_{1}}(k,k)\otimes (V^{*})^{[1]})\\
&\cong \Hom_{G}(k,\Ext^{1}_{G_{1}}(k,k)^{(-1)}\otimes V^{*}).
\end{align*}
Now by \cite[Theorem 3.1(C)(f)]{BendelNakanoPillen},  
$$\Ext^{1}_{G_{1}}(k,k)^{(-1)}\cong 
\begin{cases} 
H^{0}(\omega_{1}) & \text{$p=2$ and $\roots=C_{n}$} \\
0                             & \text{else},
\end{cases} 
$$
hence the claim holds apart from the exceptional case.  In the exceptional case,
\[
(H^0(\omega_1) \ot V^*)^G = \Hom_G(V, H^0(\omega_1)) = \Hom_G(V(\omega_1), V^*),
\]
whose dimension equals $[V^* : H^0(\omega_1)]$, so again the claim follows.
\end{proof} 

\subsection{Tilting modules}
Taking $p = 2$ in the proposition and specializing to tilting modules gives:

\begin{cor} \label{TV.p}
Let $G$ be a simple split algebraic group over a field $k$ of characteristic $2$, and let $\lambda\in X(T)_{+}$.  Assume further, in the case that $G$ has type $C$, that $[T(\la)^* : H^0(\omega_1)] = 0$.  Then
every $G$-invariant alternating bilinear form on $T(\la)$ is the polarization of a $G$-invariant quadratic form on $V$.
\end{cor} 

\begin{proof}
$T(\la)^* = T(-w_0 \la)$ has a good filtration, so Proposition \ref{TV} yields the result.
\end{proof}

\section{Exterior Powers of alternating forms} \label{C.sec}


\subsection{}
In this section, $b$ denotes a nondegenerate alternating form on a vector space $V$ of finite dimension $2n$ over a field $k$. We analyze quadratic forms on $\Lambda^r(V)$ induced by $b$ for various $r$.  This is related to quadratic forms on the fundamental representations of the group $\Sp(V,b)$ of type $C$.

It is well known that there is a bilinear form $\br$ on $\Lambda^r (V)$ defined by
\begin{equation} \label{br.def}
\br(x_1 \wedge \cdots \wedge x_r, y_1 \wedge \cdots \wedge y_r) := \det(b(x_i, y_j)_{1 \le i, j \le r}),
\end{equation}
and that $\br$ is non-degenerate since $b$ is \cite[Prop.~IX.1.9.10]{Bou:alg9}.
If $k$ has characteristic different from 2, then evidently $\br$ is symmetric for even $r$ and skew-symmetric for odd $r$.  If $k$ has characteristic 2, we consider quadratic forms, for which we have the following result which seems to be new.

\begin{prop} \label{wedge.odd} Let $\car k = 2$ and $r$ be odd. For a fixed symplectic basis $\B$ of $V$, there is a quadratic form $\qr$ on $\Lambda^r (V)$ with polar bilinear form $\br$ such that $\qr(v_1 \wedge \cdots \wedge v_r) = 0$ for $v_1, \ldots, v_r \in \B$.
\end{prop}

\begin{proof}
As $\car k = 2$, we may write out $\B$ as $e_1, \ldots, e_n, f_1, \ldots, f_n$ such that 
\begin{equation} \label{B.def}
\text{$b(e_i, e_j) = b(f_i, f_j) = 0$ for all $i, j$ and $b(e_i, f_j) = \delta_{ij}$}.
\end{equation}

Write $\Vz$ for a free $\Z$-module of rank $2n$ whose basis we denote also by $\B$ by abuse of notation; we may equally define a \emph{symmetric} bilinear form $\bz$ on 
$\Vz$ by \eqref{B.def} so that we may identify $b$ with $\bz \ot k$.  Then $\bzr$ is a symmetric bilinear form on $\Lambda^r (\Vz)$ and the map $f \!: x \mapsto \bzr(x,x)$ is a quadratic form on $\Lambda^r (\Vz)$.

We claim that $f$ always takes even values.  As $\Lambda^r (\Vz)$ is generated as an abelian group by symbols whose entries are taken from the symplectic basis, it suffices to verify that $f(x)$ is even when $x$ is such a symbol.  
But $\bzr(x,x)$ can only be nonzero for such an $x$ if for every $e_i$ in $x$ there is also a corresponding $f_i$ and vice versa.  As $r$ is odd, this is impossible and the claim is proved.

As $f$ is a homogeneous polynomial of degree 2 (with integer coefficients) in a  basis dual to $\B$, it follows that $f$ is divisible by 2.  The desired quadratic form on 
$\Lambda^r (V) = \Lambda^r (\Vz) \ot k$ is then $\qr := (\frac12 f) \ot k$.  By construction, $\qr$ has polar bilinear form $\bzr \ot k = \br$, and by the preceding paragraph $\qr$ vanishes on symbols with entries from $\B$ as desired.
\end{proof}

\begin{remark} \label{wedge.even}
In the case where $r$ is even and $2 \le r \le \dim V$, the bilinear form $\br$ is symmetric, but it is not alternating because for $s = r/2$ and $
x = e_1 \wedge f_1 \wedge e_2 \wedge f_2 \wedge \cdots \wedge e_s \wedge f_s$ we have $\br(x,x) = \pm 1$.  Therefore $x \mapsto \br(x,x)$ is a nonzero quadratic form on $\Lambda^r(V)$ which is obviously invariant under $\Sp(V, b)$.
\end{remark}

\begin{remark}
For quadratic forms over fields of characteristic $\ne 2$, the exterior powers $\Lambda^r q$ of a nondegenerate quadratic form $q$ are important invariants of $q$ and indeed form a basis for the Witt invariants of the quadratic forms of dimension $\dim q$, see \cite[p.~66]{GMS}.  However, for fields of characteristic 2, the preceding remark shows that an $\SO(V,q)$-invariant quadratic form on $\Lambda^r(V)$  for $r$ even depends only on the alternating form $b_q$ and therefore fails to capture important information about $q$.
\end{remark}

\subsection{Fundamental Weyl modules for type $C$}
As in \cite{Bou:g4}, we write $\omega_r$ for the fundamental dominant weight such that $\qform{\omega_r, \alpha_j} = \delta_{rj}$, where $\alpha_j$ is the simple root numbered $j$ in the diagram
\begin{equation} \label{C.diag}
\begin{picture}(5,0.8)
    \put(0,.4){\line(1,0){2}} 
    \put(3,.4){\line(1,0){1}}
    \put(4,.44){\line(1,0){1}}
    \put(4,0.36){\line(1,0){1}}
 
    \put(0,0){\makebox(0,0.4)[b]{$1$}}
    \put(1,0){\makebox(0,0.4)[b]{$2$}}
    \put(2,0){\makebox(0,0.4)[b]{$3$}}
    \put(3,0){\makebox(0,0.4)[b]{$n-2$}}
    \put(4,0.6){\makebox(0,0.4)[b]{$n-1$}}
    \put(5, 0){\makebox(0,0.4)[b]{$n$}}

    \multiput(2.3,0.4)(0.2,0){3}{\circle*{0.01}}
    
    \put(4.4, 0.25){\makebox(0.2,0.3)[s]{$<$}}

    \multiput(0,0.4)(1,0){6}{\circle*{\darkrad}}
\end{picture}
\end{equation}
We can see explicitly which of the fundamental Weyl modules $V(\omega_r)$ of $\Sp(V,b)$ are orthogonal.  Some of these are easy: for $r$ even, $V(\omega_r, \C)$ is orthogonal, hence so is $V(\omega_r, k)$ for every $k$.  

\begin{eg} \label{Sp.taut}
$V(\omega_1)$ is the tautological representation $V$ of $\Sp(V,b)$, and $\Sp(V,b)$ acts transitively on the nonzero vectors.  It follows that the only $\Sp(V,b)$-invariant polynomials are constant, hence $V(\omega_1, k)$ is not orthogonal for any $k$.  
\end{eg}

For the remaining cases, where $r$ is odd and $3 \le r \le n$, we identify $V(\omega_r)$ with the subspace of $\Lambda^r(V)$ generated by symbols $v_1 \wedge \cdots \wedge v_r$ such that $v_1, \ldots, v_r$ generate a totally isotropic subspace of $V$ as in \cite[\S1]{GowKleshchev} or \cite{Bruyn:Sp}.  We call such symbols \emph{of generator type}. 

\begin{prop}  \label{Sp.odd}
If $\car k = 2$, $r$ is odd, and $3 \le r \le n$, the quadratic form $\qr$ defined in Lemma \ref{wedge.odd} restricts to be $\Sp(V,b)$-invariant and nonzero on the Weyl module $V(\omega_r)$.
\end{prop}

\begin{proof}
Let $x$ be a symbol of generator type, where each entry in the symbol belongs to $\B$.  The basis $\B$ defines a pinning of $\Sp(V,b)$ as in \cite[\S{VIII.13.3}]{Bou:g7} 
or \cite[23.3]{Borel}.  We take $g \in \Sp(V,b)$ to be an element of a root subgroup relative to the pinning, so we can decompose $V$ as an orthogonal sum $V = U \perp U'$ relative to $b$ where $g(U) = U$, $g$ is the identity on $U'$, and $U$ consists of $s = 1$ or 2 of the hyperbolic planes defined by $\B$.  We write $x = y \wedge y'$ where $y \in \Lambda^t(U)$ and $y' \in \Lambda^{r - t}(U')$ are symbols of generator type and $t \le s$ because $x$ has generator type.  Writing $gy = \sum y_i$ where the $y_i$ are symbols in $\Lambda^t(U)$ with entries from $\B$, we find 
$$\qr(gx) = \qr(\sum y_i \wedge y') = \sum_{i<j} \br(y_i \wedge y', y_j \wedge y').$$  As $r \ge 3$, $r - t \ge 1$, and \eqref{br.def} shows that $\br(y_i \wedge y', y_j \wedge y') = 0$.  That is, $q(gx) = 0$.  Furthermore, if $g$ is instead taken to be in the maximal torus of $\Sp(V,b)$ defined by the pinning, then it scales $x$ and again $q(gx) = 0$.  It follows from these two calculations that $q(gx) = 0$ for all $g \in \Sp(V,b)$, hence $\qr$ vanishes on symbols of generator type, regardless of whether their entries are drawn from $\B$.

Any element $z$ of $V(\omega_r)$ can be written as $z = \sum z_i$ for $z_i$ symbols of generator type.  For any $g \in \Sp(V,b)$, we have 
$$\qr(gz) = \sum_i \qr(gz_i) + \sum_{i<j} \br(gz_i, gz_j).$$  As $gz_i$ also has generator type, $\qr(gz_i) = 0$.  Furthermore, $\br$ is canonically determined by $b$ and so is $\Sp(V,b)$-invariant, so it follows that $\qr(gz) = \sum_{i<j} \br(z_i, z_j) = \qr(z)$ as desired.
 \end{proof}


\section{Which Weyl modules are orthogonal when $\car k = 2$?}  \label{weyl.sec}

The goal of this section is to prove Theorem \ref{orth}, which determines, in case $\car k = 2$, which Weyl modules have nonzero $G$-invariant quadratic forms.  In case $\car k \ne 2$, quadratic forms are equivalent to symmetric bilinear forms and the answer is given by the material in Section~\ref{bil.weyl} and the recipe described in \cite{GW2} or \cite{Bou:g7}.

\subsection{Alternating Forms} 
Composing linear maps defined in Section~\ref{symvsym.sec}, we obtain, for vector spaces $V_1, V_2$, a linear map,
\[ 
\varpi \!: (\ot^2 V_1) \otimes  (\ot^2 V_2) \xrightarrow{\Id \ot \alt} (\ot^2 V_1) \otimes (\ot^2 V_2) \iso \ot^2 (V_1 \ot V_2) \xrightarrow{\rho} S^2(V_1 \ot V_2),
\] 
where the middle isomorphism is $v_1 \ot v'_1 \ot v_2 \ot v'_2 \mapsto v_1 \ot v_2 \ot v'_1 \ot v'_2$.  Replacing the first map $\Id \ot \alt$ with $\alt \ot \Id$ does not change $\varpi$.
Put $e_i \!: \ot^2 V_i \twoheadrightarrow \Lambda^2(V_i)$ for the quotient map.  

\begin{lemma}[Tignol's product] \label{tignol}
The map $\varpi$ vanishes on $\ker (e_1 \ot e_2)$ and induces a linear map
\[
\Lambda^2(V_1) \ot  \Lambda^2(V_2) \to S^2(V_1 \ot V_2).
\]
\end{lemma}

\begin{proof}
For $i = 1, 2$, the map $\alt$ vanishes on the subspace $\I''_2(V_i)$ of $\ot^2 V_i$ spanned by elements $v \ot v$ for $v \in V_i$, so $\varpi$ vanishes on $(\ot^2 V_1) \ot \I''_2(V_2) = \ker(\Id \ot e_2)$.  Because replacing the first map in the definition of $\varpi$ does not change $\varpi$, $\varpi$ also vanishes on $\I''_2(V_1) \ot (\ot^2 V_2)$,  and we conclude that $\varpi$ vanishes on the kernel of $e_1 \ot e_2$.
\end{proof}

The following proposition in the case $\car k = 2$ can be found in \cite[\S3]{SinWillems}.  Our proof invokes Tignol's product from Lemma \ref{tignol}.
\begin{prop} \label{tensor}
If $V_1$, $V_2$ are $k$-vector spaces with alternating bilinear forms $b_1, b_2$, then there is a unique quadratic form $q$ on $V_1 \ot V_2$ so that
\begin{equation} \label{tensor.1}
q(\sum_i v_{1i} \ot v_{2i}) = \sum_{i<j} b_1(v_{1i}, v_{1j}) \, b_2(v_{2i}, v_{2j}) \quad \text{for $v_{\ell i} \in V_\ell$.}
\end{equation}
If $b_1$ and $b_2$ are nondegenerate, then so is $q$.  If $b_i$ is invariant under a group $G_i$, then $q$ is invariant under $G_1 \times G_2$.
\end{prop}
  
\begin{proof} The alternating form $b_i$ determines an element of $\Lambda^2(V^*_i)$.  Plugging $b_1 \ot b_2$ into Tignol's product gives an element $q \in S^2(V^*_1 \ot V^*_2) = S^2((V_1 \ot V_2)^*)$, i.e., with a quadratic form on $V_1 \ot V_2$.
\end{proof}

\begin{cor} \label{sum}
Let $\la, \mu \in X(T)_+$.
If  the Weyl modules $V(\la), V(\mu)$ have nonzero $G$-invariant alternating bilinear forms, then $V(\la + \mu)$ is orthogonal.
\end{cor}

\begin{proof}
The tensor product $V(\la) \ot V(\mu)$ has $\la + \mu$ as an extreme weight, and so there is a nonzero $G$-equivariant map $\pi \!: V(\la + \mu) \to V(\la) \ot V(\mu)$.
The hypothesis on $V(\la)$ and $V(\mu)$ gives a $G$-invariant quadratic form $q$ on the tensor product by Proposition ~\ref{tensor}, hence $q\pi$ is a $G$-invariant quadratic form on $V(\la + \mu)$, and it suffices to check that $q \pi$ is nonzero.

Fix highest weight vectors $x^+, y^+$ and lowest weight vectors $x^-, y^-$ of $V(\la), V(\mu)$ respectively; note that the bilinear forms are nonzero on the pairs $(x^+, x^-)$ and $(y^+, y^-)$ as in Lemma \ref{bil}.  The sum of the highest and lowest weight spaces of $V(\la + \mu)$ is identified via $\pi$ with $k(x^+ \ot y^+) + k(x^- \ot y^-)$, and $q(x^+ \ot y^+ + x^- \ot y^-)$ is nonzero by \eqref{tensor.1}.
\end{proof}


\subsection{} We continue by providing below some examples for symplectic groups. 

\begin{eg} \label{Sp.bad}
Let $G$ be a quotient of $T_0 \times \prod_i G_i$ as in \ref{red.gen} such that $G_j \cong \Sp_{2n}$ for some $n \ge 1$ and some $j$, and $\la \in X(T)_+$ such that $\la$ restricts to be zero on $T_0$ and $G_i$ for $i \ne j$, but on $G_j$ it is $\omega_1$.  Then by Lemma \ref{red.tensor}, the restriction of $V(\la)$ to $G_j$ is the tautological representation as in Example \ref{Sp.taut} and the restriction to $T_0$ and to $G_i$ for $i \ne j$ is a trivial representation.  Therefore, $k[V(\la)]^G = k[V(\omega_1)]^{\Sp_{2n}} = k$ and $V(\la)$ is symplectic (i.e., $\Lambda^2(V(\la)^*)^G \ne 0$) and not orthogonal (i.e., $(\Quad V(\la))^G = 0$).
\end{eg}

\begin{theorem} \label{orth}
Let $G$ be a split reductive group over a field $k$ of characteristic 2 and let $\la \in X(T)_+$ be nonzero.
Then exactly one of the following holds:
\begin{enumerate}
\item \label{orth.orth} $V(\la)$ is orthogonal (i.e., $(\Quad V(\la))^G \ne 0$), or
\item \label{orth.bad} $G$ and $V(\la)$ are as in Example \ref{Sp.bad}, or
\item \label{orth.ne} $\la \ne -w_0 \la$.
\end{enumerate}
Furthermore, if \eqref{orth.orth} occurs, then $(\Quad V(\la))^G = kq$ for a quadratic from $q$ with $\rrad b_q = \rrad V(\la)$ and  one of the following holds:
\begin{enumerate}
\renewcommand{\theenumi}{\roman{enumi}}
\item \label{N.caseZ}  $\rrad q = \rrad V(\la)$, $(\Quad L(\la))^G$ is 1-dimensional and spanned by the quadratic form $\qb$ induced by $q$, and $\rrad b_{\qb} = 0$; or
\item \label{N.caseN} $(\Quad L(\la))^G = 0$, and $\rrad q$ is a codimension-1 $G$-submodule of $\rrad V(\la)$.
\end{enumerate}

\end{theorem}

%
%

\begin{proof}
Suppose first that \eqref{orth.bad} and \eqref{orth.ne} are false; we prove \eqref{orth.orth}.
As in Section~\ref{red.gen}, without loss of generality we can assume that $G$ is of the form $T_0 \times \prod_{i=1}^n G_i$ where $T_0$ is a split torus and the $G_i$ are simple, split, simply connected algebraic groups over $k$, and write $\la = \sum_{i=0}^n \la_i$.  As $\la = -w_0 \la$ and $w_0$ is the product of the longest element in the Weyl groups for each of the $G_i$, it follows that $\la_i$ has the same property for all $i$ and in particular that $\la_0 = 0$.  If $\la_i \ne 0$, then $V(\la_i)$ has a nonzero alternating bilinear form by Lemma \ref{bil}.  Hence if two or more $\la_i$'s are nonzero, Corollary \ref{sum} and Lemma \ref{red.tensor} combine to give \eqref{orth.orth}.

So assume $\la = \la_1$.  As $T_0$ and $G_i$ for $i \ne 1$ act trivially on $V(\la)$, we may assume that $G = G_1$, i.e., that $G$ is simple.  By hypothesis, $(\Bil V(\la))^G \ne 0$, and Proposition \ref{TV} completes the proof of \eqref{orth.orth}.

\smallskip

Now suppose that \eqref{orth.orth} holds.
Because $\la \ne 0$, for $V = V(\la)$ or $L(\la)$, $V$ lacks a codimension-1 submodule, hence polarization gives an injection $(\Quad V)^G \hookrightarrow \Lambda^2(V^*)^G$ by Corollary \ref{symvsym}\eqref{symvsym.inj}, thus \eqref{orth.orth} implies not \eqref{orth.ne} completing the proof of the first claim,  and $\dim (\Quad V)^G \le 1$. 

Let $q \in (\Quad V(\la))^G$ be nonzero.  If $\rrad b_q = \rrad q$, then the remaining claims in \eqref{N.caseZ} are obvious, so 
suppose $q\vert_{\rrad V(\la)}$ is not identically zero.  Recall from Section~\ref{back.sec} that $G$ and $V(\la)$ are defined over $\FF_2$.  The natural map $\FF_2[V(\la, \FF_2)]^G \ot k \to k[V(\la, k)]^G$ is an isomorphism \cite{Sesh:GR}, so $q$ is the base change from $\FF_2$ of a $G$-invariant quadratic form $q_0$ on $V(\la, \FF_2)$.  We take $X_0$ to be the kernel of $q_0$ restricted to the radical of $V(\la, \FF_2)$; it has codimension at most one because $\FF_2$ is perfect.  Then the radical $X$ of $q$ is a proper subspace of $\rrad V(\la)$ and contains $X_0 \ot k$, hence $\dim(\rrad V(\la)/X) = 1$, and the claim is proved.
\end{proof}

\subsection{} In addition to Proposition \ref{Sp.odd} above, in some of the other cases where $V(\la, \C)$ is symplectic yet $V(\la, k)$ is orthogonal for $\car k = 2$, the $G$-invariant quadratic form has been described in the literature.  For the half-spin representations of type $D$, see \cite[p.~150]{KMRT}.  For the representation $\Lambda^r(L(\omega_{1}))\cong L(\omega_{r})$ of $\SL_{2r}$, see \cite[10.12]{KMRT}. And for the 56-dimensional representation $V(\omega_7)$ of $E_7$, one notes that the well-known $E_7$-invariant quartic on $V(\omega_7)$  is defined over $\Z$ and becomes a square after reduction mod 2 --- see \cite[p.~87]{Brown:E7} for an explicit formula.
 
\begin{example}
For $G = \SO_{2n}$, write $\omega_{1}$ for the highest weight of the tautological representation $V(\omega_{1})$.  As in Section \ref{killing.2}, $V(\at)$ is the Lie algebra of the simply connected cover $\Spin_{2n}$.  The weight $\at$ is a maximal weight of $\Lambda^2(V(\omega_{1}))$, and so there is a nonzero $G$-equivariant function $f \!: V(\at) \to \Lambda^2(V(\omega_{1}))$ unique up to a scalar multiple.  If $\car k \ne 2$, $f$ is an isomorphism, see e.g.~\cite[Lemma 45.3(2)]{KMRT}.  However, if $\car k = 2$, then the radical $\rrad V(\at)$ is the 2-dimensional center $\Lie(\mu_2 \times \mu_2)$ of the Lie algebra; as the natural $G$-invariant bilinear form on $\Lambda^2(V(\omega_{1}))$ given by \eqref{br.def} is nondegenerate, it follows that $\ker f = \rrad V(\at)$.
\end{example}

\subsection{Meaning of orthogonality for Weyl modules}
If $L(\la)$ is orthogonal, then the action of $G$ on $L(\la)$ is a homomorphism $G \to \SO(q)$ for some $G$-invariant quadratic form $q$, and $\SO(q)$ is a semisimple group of type $B$ or $D$.  We now describe the relevant group schemes in the case of a $G$-invariant quadratic form $q$ on the Weyl module $V(\la)$.

Suppose for the moment that $f$ is a polynomial function on a vector space $V$ and $U < V$ is a subspace such that $f(u + v) = f(v)$ for all $u \in U$ and $v \in V$.  Then $f$ induces canonically a polynomial function $\fb$ on $\Vb := V/U$, and we define $\O(\fb)$ to be the closed subgroup scheme of $\GL(\Vb)$ stabilizing $\fb$.  Similarly, set $\O(f,U)$ to be the sub-group-scheme of $\GL(V)$ leaving both $f$ and $U$ invariant.

\begin{lem}
In the notation of the preceding paragraph, there is a short exact sequence of group schemes
\[
\begin{CD}
1 @>>> \Hom(V/U, U) @>>> \O(f, U) @>>> \O(\fb) \times \GL(U) @>>> 1.
\end{CD}
\]
\end{lem}

\begin{proof}
Choosing any complement $U'$ of $U$ in $V$ and writing linear transformations of $V$ in terms of a basis adapted to the decomposition $V = U \oplus U'$, we find that for every $k$-algebra $S$, 
\[
\O(f, U)(S) = 
\begin{pmatrix}
\GL(U)(S) & \Hom(U' \ot S, U \ot S) \\
0 & \O(\fb)(S)
\end{pmatrix}.  \qedhere
\]
\end{proof}

Now suppose that $q$ is a quadratic form on $V$, and that $\dim (\rrad b_q / \rrad q) \le 1$.  This holds if $\car k \ne 2$, or if $k$ is perfect, or by Theorem \ref{orth} if $V$ is a Weyl module for a split reductive group $G$ and $q$ is $G$-invariant.  Then we take $U := \rrad q$; the hypothesis on the dimension assures us that for every extension $K$ of $k$, $\rrad (q \ot K) = (\rrad q) \ot K = U \ot K$, and we find that $\O(q) = \O(q, U)$ as group schemes.  Therefore, for $\qb$ the induced quadratic from on $V/U$, we have an exact sequence
\[
\begin{CD}
1 @>>> \Hom(V/U, U) @>>> \O(q) @>>> \O(\qb) \times \GL(U) @>>> 1.
\end{CD}
\]
We define $\SO(q)$ to be the fiber of $\SO(\qb) \times \SL(U)$ in $O(q)$.  Clearly, the action of $G$ on $V$ preserving $q$ gives homomorphisms $G \to \SO(q) \to \SO(\qb)$, where $\SO(\qb)$ is semisimple of type $B$ or $D$.

\section{Quadratic forms on  irreducible representations} \label{irred.sec}

\subsection{} For what $\la \in X(T)_+$ is $L(\la)$ orthogonal?  By Theorem \ref{orth}, $L(\la)$ of $G$ is orthogonal if and only if (a) the Weyl module $V(\la)$ is orthogonal and (b) a nonzero $G$-invariant quadratic form $q$ on $V(\la)$ vanishes on $\rrad V(\la)$.  As the case where $\car k \ne 2$ is treated by Lemma \ref{bil}, we assume here that $\car k = 2$, and we know the answer to (a) by Theorem \ref{orth}.  

For question (b), we have the following sufficient conditions.
Let $W_{p}=W\ltimes p{\mathbb Z}\Phi$ be the affine Weyl group, which acts on $X(T)$ via the dot action. For $w\in W_{p}$ and $\lambda\in X(T)$, the action is denoted by $w\cdot \lambda$. 
The following result provides sufficient conditions to guarantee that $L(\lambda)$ is orthogonal. 

\begin{prop} \label{N.crit}
Suppose $\car k = 2$, $\la \ne 0$, and $V(\la)$ is orthogonal.  Among the statements
\begin{enumerate}
\item \label{N.pos} $\lambda$ is not a sum of positive roots; 
\item \label{N.root} $\la\notin W_2\cdot 0$; 
\item  \label{N.coh} $\operatorname{H}^1(G, L(\la)) = 0$; 
\item \label{N.case} $L(\la)$ is orthogonal,
\end{enumerate}
the following implications hold: \eqref{N.pos} $\Rightarrow$ \eqref{N.root} $\Rightarrow$ \eqref{N.coh} $\Rightarrow$  \eqref{N.case}.
\end{prop}

\begin{proof} 
Using \cite[II.2.14, II.2.12(4)]{Jantzen} we have
\[
\operatorname{H}^1(G, L(\la)) = \Ext^1_G(k, L(\la)) = \Hom_G(\rrad V(\la), k).
\]
If \eqref{N.case} fails, then by Theorem \ref{orth}\eqref{N.caseN} $\Hom_G(\rrad V(\la), k) \ne 0$, i.e., \eqref{N.coh} fails.  If \eqref{N.coh} fails, then $L(0)$ is a factor in the composition series for $V(\la)$, hence $\la\in W_2\cdot 0$ and \eqref{N.root} fails.  If $\la$ is in $W_2 \cdot 0$, then $\la$ is in the root lattice; since $\la$ is a dominant weight, it is a sum of positive roots, so \eqref{N.pos} implies \eqref{N.root}.
\end{proof}

It is not hard to find $\la \in Q$ but $\la \not\in W_p \cdot 0$, i.e., an example to show that \eqref{N.root} $\not\Rightarrow$ \eqref{N.pos}.  The examples in Section \ref{adj.eg} shows that the converses of each of the other implications can fail. Also, note that condition \eqref{N.root} in Proposition~\ref{N.crit} can be 
replaced by the statement that ``$0$ is not strongly linked to $\lambda$'', for strong linkage as described in \cite[II Chapter 6]{Jantzen}.

\begin{example}
Let $\omega$ be a dominant weight of $\Sp_{2n}$ that is neither a sum of positive roots nor $\omega_1$.  For every field of characteristic 2, the Weyl module $V(\omega)$ of $\Sp_{2n}$ is orthogonal by Theorem \ref{orth}, and Proposition \ref{N.crit} says that the irreducible representation $L(\omega)$ is orthogonal.
\end{example}

\subsection{Adjoint representations}  \label{adj.eg} Suppose now that $G$ is a split simple group over a field $k$ of characteristic 2. The highest weight of the adjoint representation is $\at$.  We illustrate the proposition by determining whether the 
irreducible representation $L(\at)$ is orthogonal, i.e., whether the reduced Killing form $s$ defined in Example \ref{killing.2} vanishes on $\rrad V(\at)$.  According to the description of $s$ in that example, it suffices to first find the radical $U$ of the polar bilinear form of $q^\vee \ot k$, which is the kernel of the linear transformation $DC$, and then to check if $q^\vee \ot k$ as given by \eqref{q.def} is identically zero on $U$.  These steps involve linear algebra with explicit matrices over $\FF_2$ and explicit formulas can be found in \cite{BabicCh}, so we merely summarize the results in Table \ref{adjoint.eg}.  
\begin{table}[hbt]
\begin{tabular}{cc|cc} 
$\Phi$&restrictions&$L(\at)$ orthogonal? &$\dim \operatorname{H}^1(G, L(\at))$ \\ \hline
&$n \equiv 0, 2 \bmod 4$&yes&0 \\
$A_n$&$n \equiv 1 \bmod 4$&no&1 \\
&$n \equiv 3 \bmod 4$&yes&1 \\ \hline

&$n \equiv 0 \bmod 4$&yes&1 \\
$B_n$ ($n \ge 3$)&$n \equiv 1, 3 \bmod 4$&yes&0 \\
&$n \equiv 2 \bmod 4$&no&1 \\ \hline

&$n \equiv 0 \bmod 4$&yes &2 \\ 
$D_n$ ($n \ge 4$)&$n \equiv 1, 3 \bmod 4$&yes&1 \\
&$n \equiv 2 \bmod 4$&no &2 \\ \hline
$C_n$ or $E_7$&&no&1 \\ 
$E_6$, $E_8$, $F_4$, or $G_2$&&yes&0 
\end{tabular}
\vskip .5cm 
\caption{Orthogonality of $L(\at)$ for a simple group $G$ over a field of characteristic 2} \label{adjoint.eg}
\end{table}

For types other than $C$, the orthogonality of $L(\at)$ has been determined in \cite[\S3]{GowWillems:methods} by a somewhat different argument. Although for type $C$, we find that $L(\at)$ is never orthogonal, which appears to contradict the final sentence of \cite{GowWillems:methods}: ``Our methods can also be applied to the Lie algebra of type $C_l$, but we omit a statement of our findings as they correspond to those obtained for $B_l$.''

In the table, for the convenience of the reader, we also give $\dim \H^1(G, L(\at))$.  This amounts to identifying the cohomology group with $\Hom_G(\rrad V(\at), k)$ and consulting the description of the $G$-module structure of 
$V(\at, k)$ given in \cite{Hiss}.  Note that the case of $A_n$ with $n \equiv 0 \bmod{4}$ shows that \eqref{N.coh}  $\not\Rightarrow$ \eqref{N.root} in Proposition \ref{N.crit} and with $n \equiv 3 \bmod{4}$ shows that \eqref{N.case} $\not\Rightarrow$ \eqref{N.coh}.

\section{Cohomology of $\Lambda^2$ of Irreducible and Weyl modules} \label{wedge2.coh}

\subsection{} We begin with a lemma which gives conditions on $V$ that guarantee
$\text{H}^1(G, V \ot V) = 0$.  Recall that $G$ is assumed to be split reductive.

\begin{lemma} \label{tensorlemma}
Let $V$ be a module admitting a good filtration or $V=L(\lambda)$ with $\la \in X(T)_+$. Then $\operatorname{H}^1(G, V \ot V) = 0$.
\end{lemma}

\begin{proof} If $V$ has a good filtration then $V\otimes V$ has a good filtration and the result follows, because $\H^1(G, H^0(\sigma_1) \ot H^0(\sigma_2)) = 0$ for all $\s_1, \s_2 \in X(T)_+$ by \cite[II 4.13 Proposition]{Jantzen}.

Now suppose that 
$V=L(\lambda)$, and suppose that $\H^1(G, V \ot V)$, i.e., $\Ext^{1}_{G}(L(-w_{0}\lambda),L(\lambda))$, is not zero. 
According to \cite[II 2.14 Remark]{Jantzen}, one has either $-w_{0}\lambda < \lambda$ or $\lambda < -w_{0}\lambda$, contradicting Lemma \ref{comparable}. 
\end{proof}

\subsection{} We now calculate $\H^1(G, \Lambda^2(V))$ for some of the modules $V$ we have studied.
\begin{theorem} \label{wedge2}
Let $G$ be a split reductive group and $\la \in X(T)_+$.  Then
\[
\operatorname{H}^1(G, \Lambda^2(V)) = \begin{cases}
0 & \text{if $\car k \ne 2$ and $V = T(\la)$, $L(\la)$, or $H^0(\la)$;} \\
(\Quad V^*)^G & \text{if $\car k = 2$, $V = L(\la)$ or $H^0(\la)$, and $\la \ne 0$.} 
\end{cases}
\]
\end{theorem}

Evidently, if $\la = 0$, then for $V = L(\la) = H^0(\la) = k$, we have $(\Quad V)^G = k$ but $\Lambda^2(V) = 0$, so $\H^1(G, \Lambda^2(V)) = 0$.

\begin{proof} 
Taking $G$-fixed points in \eqref{wedgeS.seq} gives an exact sequence
\[
\begin{CD}
\Lambda^2(V)^G @>\alt^G>> \Bil(V^*)^G @>\rho^G>> \Quad(V^*)^G @>>> \text{H}^1(G, \Lambda^2(V)) @>>> \H^1(G, V \ot V),
\end{CD}
\]
where 
the last term is zero for $V = H^0(\la)$, $T(\la)$, or $L(\la)$ by Lemma \ref{tensorlemma}.  If $\car k \ne 2$, then $\rho^G$ is surjective and the claims follow.

So suppose $\car k = 2$ and $V = H^0(\la)$ or $L(\la)$ with $\la \ne 0$.  If $(\Quad V^*)^G = 0$, we are done, so assume otherwise.  Then Lemma \ref{bil} implies that $\rho^G$ is the zero map, and again we are done.
\end{proof}

\begin{eg} \label{quadstar}
Suppose $V = L(\la)$ for some $\la \in X(T)_+$; we claim that $(\Quad V^*)^G \cong (\Quad V)^G$.  Indeed,
if $\la = -w_0 \la$, then $V \cong V^*$ and the claim is trivial.  On the other hand, if $\la \ne -w_0 \la$, then $(\Bil V)^G = 0$, hence $(\Quad V)^G$ consists of totally singular quadratic forms, which are necessarily zero (by Theorem ~\ref{orth} because $\la \ne 0$); the same applies to $V^* = L(-w_0 \la)$.
\end{eg}

Sin and Willems showed \cite[Proposition 2.7]{SinWillems}, in the case where $\car k = 2$ and $V = L(\la)$, that $\H^1(G, \Lambda^2 (V)) = 0$ implies that $(\Quad V)^G = 0$.  We are able to say more in Theorem \ref{wedge2} thanks to Lemma~\ref{tensorlemma}. 

\subsection{} By combining Theorem \ref{wedge2} with Theorem \ref{orth}, we obtain the following first cohomology calculation for all induced modules. 

\begin{cor}
Let $G$ be a split reductive group and $\la \in X(T)_+$.  Then
\[
\operatorname{H}^1(G, \Lambda^2(H^0(\la))) = \begin{cases}
k & \text{if $\car k = 2$ and $0 \ne \la = -w_0 \la$,  but} \\
    & \text{$G$ and $V(\la)$ are not as in Example \ref{Sp.bad};}\\
0 & \text{otherwise.}
\end{cases}
\]
\end{cor}

\bibliographystyle{amsalpha}
\bibliography{skip_master}

\end{document}